\documentclass[12pt]{amsart}
\usepackage{amsmath}
\usepackage{amssymb}
\usepackage{amsfonts}
\usepackage{amsthm}
\usepackage{mathrsfs}
\usepackage[all]{xy}
\usepackage[pdftex]{graphicx}
\usepackage{color}
\usepackage{cite}

\begin{document}

\theoremstyle{plain}
\newtheorem{thm}{Theorem}
\newtheorem{lemma}[thm]{Lemma}
\newtheorem{cor}[thm]{Corollary}
\newtheorem{conj}[thm]{Conjecture}
\newtheorem{prop}[thm]{Proposition}
\newtheorem{heur}[thm]{Heuristic}

\theoremstyle{definition}
\newtheorem{defn}[thm]{Definition}
\newtheorem*{ex}{Example}
\newtheorem*{notn}{Notation}

\title[Convergence Rates of Gibbs Samplers]{Analysis of Convergence Rates of some Gibbs Samplers on Continuous State Spaces}
\author{Aaron Smith}

\address{Department of Mathematics, Stanford University, Stanford, CA 94305}
\email{asmith3@math.stanford.edu}
\date{\today}
\maketitle

\section{Abstract}

We use a non-Markovian coupling and small modifications of techniques from the theory of finite Markov chains to analyze some Markov chains on continuous state spaces. The first is a generalization of a sampler introduced by Randall and Winkler, the second a Gibbs sampler on narrow contingency tables.

\section{Introduction}

The problem of sampling from a given distribution on high-dimensional continuous spaces arises in the computational sciences and Bayesian statistics, and a frequently-used solution is Markov chain Monte Carlo (MCMC); see \cite{Liu01} for many examples. Because MCMC methods produce good samples only after a lengthy mixing period, a long-standing mathematical question is to analyze the mixing times of the MCMC algorithms which are in common use. Although there are many mixing conditions, the most commonly used is called the mixing time, and is based on the total variation distance: \par
For measures $\nu$, $\mu$ with common measurable $\sigma$-algebra $\mathcal{A}$, the \textit{total variation distance} between $\mu$ and $\nu$ is
\begin{equation*}
\vert \vert \mu - \nu \vert \vert_{TV} = \sup_{A \in \mathcal{A}} \left( \mu(A) - \nu(A) \right)
\end{equation*} \par 
For an ergodic discrete-time Markov chain $X_{t}$ with unique stationary distribution $\pi$, the \textit{mixing time} is
\begin{equation*}
\tau(\epsilon) = \inf \{t : \, \vert \vert \mathcal{L}(X_{t}) - \pi \vert \vert_{TV} < \epsilon \}
\end{equation*} \par

Although most scientific and statistical uses of MCMC methods occur in continuous state spaces, much of the mathematical mixing analysis has been in the discrete setting. The methods that have been developed for discrete chains often break down when used to analyze continuous chains, though there are some efforts, such as \cite{Yuen01} \cite{Rose95} \cite{LoVe03}, to create general techniques. This paper extends the author's previous work in \cite{Smit11} and work of Randall and Winkler \cite{RaWi05a}, and attempts to provide some more examples of relatively sharp analyses of continuous chains similar to those used to develop the discrete theory. \par 

The first process that we analyze is a Gibbs sampler on the simplex with a very restricted set of allowed moves. Fix a finite group $G$ of size $n$ with symmetric generating set $R$ of size $m$, with $id \notin R$. For unity of notation, label the group elements with the integers from 1 to $n$. We consider the process $X_{t}[g]$ on the simplex $\Delta_{G} = \{ X \in \mathbb{R}^{n} \, \vert \sum_{g \in G} X[g] = 1; X[g] \geq 0 \}$. At each step, choose $g \in G$, $r \in R$ and $\lambda \in [0,1]$ uniformly, and set 

\begin{align} \label{EqCayleyMoveRep}
X_{t+1}[g] &= \lambda (X_{t}[g] + X_{t}[gr]) \nonumber \\
X_{t+1}[gr] &= (1 - \lambda) (X_{t}[g] + X_{t}[gr]) 
\end{align} \par 

For all other $h \in G$ set $X_{t+1}[h] = X_{t}[h]$. Let $U_{G}$ be the uniform distribution on $\Delta_{G}$; this is also the stationary distribution of $X_{t}$. Also consider a random walk $Z_{t}$ on $G$, where in each stage we choose $g \in G$ and $r \in R$ uniformly at random and set $Z_{t+1} = gr$ if $Z_{t} = g$, set $Z_{t+1} = g$ if $Z_{t} = gr$, and $Z_{t+1} = Z_{t}$ otherwise. This is the standard simple random walk on the Cayley graph, slowed down by a factor of about $n$. Let $\widehat{\gamma}$ be the spectral gap of the walk $Z_{t}$, and follow the notation that $\mathcal{L}(X)$ denotes the distribution of a random variable $X$.

\begin{thm} [Convergence Rate for Gibbs Sampler with Geometry] \label{ThmConvRateCayley}
For $T >  \frac{8C}{\widehat{\gamma}} \log(n)$, $C > \frac{103}{4}$, and $n$ satisfying $n > \max \left( 4096, \frac{C}{3} + \frac{10}{3}, \left( \frac{C}{3} - \frac{13}{12} \right) \log(n) \right)$
\begin{equation*}
\vert \vert \mathcal{L}( X_{T}) - U_{G} \vert \vert_{TV} \leq 7 n^{4.5 -\frac{C}{6}}
\end{equation*}
and conversely for $T < \frac{k}{\widehat{\gamma}}$,
\begin{equation*}
\vert \vert \mathcal{L} (X_{T}) - U_{G} \vert \vert_{TV} \geq \frac{1}{2} e^{-k} - 4 n^{-\frac{1}{3}}
\end{equation*}
\end{thm}

This substantially generalizes \cite{RaWi05a} and \cite{Smit11}, from samplers corresponding to $G = \mathbb{Z}_{n}$, and $R = \{ 1, -1 \}$ or $R = \mathbb{Z}_{n} \backslash \{ 0 \}$ respectively, to general Cayley graphs. In addition to being of mathematical interest, this process is an example of a gossip process with some geometry, studied by electrical engineers and sociologists interested in how information propagates through networks; see \cite{Shah09} for a survey. \par 

The proof of the upper bound will use an auxilliary chain similar to that found in \cite{RaWi05a}, a coupling argument improved from \cite{Smit11}, and an unusual use of comparison theory from \cite{DiSa93}. The proof of the lower bound is elementary.\par 

The next example consists of narrow contingency tables. Beginning with the work of Diaconis and Efron \cite{DiEf85} on independence tests, there has been interest in finding efficient ways to sample uniformly from the collection of integer-valued matrices with given row and column sums. A great deal of this effort has been based on Markov chain Monte Carlo methods. While some of the efforts have dealt directly with Markov chains on these integer-valued matrices, much recent success, including \cite{DKM97} \cite{Morr02}, has involved using knowledge of Gibbs samplers on convex sets in $\mathbb{R}^{n}$ and clever ways to project from the continuous chain to the desired matrices \cite{Morr00}. \par 

Unfortunately, while the general bounds are polynomial in the number of entries in the desired matrix, they often have a large degree and leading coefficient; see \cite{Lova98}. In this paper, we find some better bounds for very specific cases. Like the paper \cite{Smit11}, this is part of an attempt to make further use of non-Markovian coupling techniques \cite{HaVi03} \cite{Borm11} \cite{BuKo11} \cite{Matt87} and also to expand the small set of carefully analyzed Gibbs samplers \cite{RaWi05a} \cite{RaWi05b} \cite{DKSC08} \cite{DKSC10}. \par 

 We consider the following Gibbs sampler $X_{t}[i,j]$ on the space $M_{n} = \{ X \in \mathbb{R}^{2n}: \sum_{i=1}^{n} X[i, j] = n \, \,  \forall \, \, 1 \leq j \leq 2, \sum_{j=1}^{2} X[i, j] = 2 \, \, \forall \, \, 1 \leq i \leq n, X[i,j] \geq 0  \}$ of nonnegative $n$ by $2$ matrices with column sums fixed to be $n$ and row sums fixed to be $2$. To make a step of the Gibbs sampler, choose two distinct integers $1 \leq i < j \leq n$ and update the four entries $X_{t+1}[i,1]$, $X_{t+1}[i,2]$, $X_{t+1}[j,1]$ and $X_{t+1}[j,2]$ to be uniform conditional on all other entries of $X_{t}$. Let $U_{n}$ be the uniform distribution on $M_{n}$ inherited from Lebesgue measure. Then we find the following reasonable bound on the mixing time of this sampler:

\begin{thm} [Convergence Rate for Narrow Matrices]  Fix $C > 23$ and set $a = \frac{2}{11}(C + 18.25)$. Then, for $n > \max(4096, \frac{2}{11}C + \frac{75}{11})$ satisfying $\frac{n}{\log(n)} > \frac{6(2a-13)(a-7)}{2a - 15}$ and $T > C n \log(n) $, 
\begin{equation*}
\vert \vert \mathcal{L}(X_{T}) - U_{n} \vert \vert_{TV} \leq 11 n^{-\frac{C}{11} + \frac{19}{44}}
\end{equation*}
and conversely for fixed $0<C<1$ with $n$ sufficiently large and $T <  (1-C)n \log(n)$, 
\begin{equation*}
\vert \vert \mathcal{L}(X_{T}) - U_{n} \vert \vert_{TV} \geq 1 - 2n^{-k}
\end{equation*}
\end{thm}

\section{General Strategy and the Partition Process} \label{SecPartProc}
Both of our bounds will be obtained using a similar strategy, ultimately built on the classical coupling lemma. We recall that a coupling of Markov chains with transition kernel $K$ is a process $(X_{t}, Y_{t})$ so that marginally both $X_{t}$ and $Y_{t}$ are Markov chains with transition kernel $K$. Although we always couple entire paths $\{ X_{t} \}_{t = 0}^{T}$ and $\{ Y_{t} \}_{t = 0}^{T}$, we often use the shorthand notation of saying that we are coupling $X_{t}$ and $Y_{t}$. In order to describe a coupling, note that for both walks being studied, the evolution of the Markov chain $X_{t}$ can be represented by $X_{t+1} = f(X_{t}, i(t), j(t), \lambda(t))$, where $f$ is a deterministic function, $i(t), j(t)$ are random coordinates (either elements of $[n]$ or of a group $G$), and $\lambda(t)$ is drawn from Lebesgue measure on $[0,1]$. These representations are given in equations \eqref{EqNarrowRepDeltaPos}, \eqref{EqNarrowRepDeltaNeg} and \eqref{EqCayleyMoveRep} respectively. To couple $X_{t}$ and $Y_{t}$, it is thus enough to couple the update variables $i(t)^{\alpha}, j(t)^{\alpha}, \lambda(t)^{\alpha}$, with $\alpha \in \{ x, y \}$, used to construct $X_{t}$ and $Y_{t}$ respectively. \par

Our couplings will provide bounds on mixing times through the following lemma (see \cite{LPW09}, Theorem 5.2 - they work in discrete space, but their proof doesn't rely on this assumption):
\par
\begin{lemma} [Fundamental Coupling Lemma] If $(X_{t}, Y_{t})$ is a coupling of Markov chains, $Y_{0}$ is distributed according to the stationary distribution of $K$, and $\tau$ is a random time with the property that $X_{t} = Y_{t}$ for $t \geq \tau$, then 
\begin{equation*}
\vert \vert \mathcal{L}(X_{t})- \mathcal{L}(Y_{t}) \vert \vert_{TV} \leq P[\tau > t]
\end{equation*}
\end{lemma}
\par

In each chain, then, we begin with $X_{t}$ started at a distribution of our choice, and $Y_{t}$ started at stationarity. For any fixed (large) $T$, we will then couple $X_{t}$ and $Y_{t}$ so that they will have coupled by time $T$ with high probability. Each coupling will have two phases: an initial phase from time $0$ to time $T_{1}$ in which $X_{t}$ and $Y_{t}$ get close with high probability, and a non-Markovian coupling phase from time $T_{1}$ to time $T = T_{1} + T_{2}$ in which they are forced to collide. Unlike many coupling proofs, the time of interest $T$ must be specified before constructing the coupling.  \par
While the initial contraction phases are quite different for the two chains, the final coupling phase can be described in a unified way. The unifying device is the partition process $P_{t}$ on set partitions of $[n] = \{1, 2, \ldots, n \}$, introduced in \cite{Smit11} for a special case of the first sampler treated here. This partition process contains some information about the coordinates $\{ i(t), j(t) \}_{t = T_{1}}^{ T-1}$ used by $Y_{t}$ throughout the entire process, and is the only source of information from the future that is used to construct the non-Markovian coupling. Critically, we don't use any information about the random variables $\lambda(t)$ used at each step, which makes it trivial to check that the couplings constructed in this paper have the correct marginal distributions. \par 
The process $\{ P_{t} \}_{t=T_{1}}^{T}$ will consist of a set of nested partitions of $[n]$, $P_{T_{1}} \leq P_{T_{1} + 1} \leq \ldots \leq P_{T} = \{ \{1 \}, \{ 2 \}, \ldots, \{ n \} \}$, where we say partition $A$ is less than partition $B$ if every element of partition $B$ is a subset of an element of partition $A$. To construct $P_{t}$, we first look at the sequence of graphs $G_{t}$ with vertex set $[n]$ and edge set $\{ (i(s), j(s)) \, : \, s \geq t \} $. Then let $P_{t}$ consist of the connected components of $G_{t}$. While constructing $P_{t}$, we will also record a series of `marked times' $ T-1 = t_{1} > t_{2} > \ldots > t_{m}$ and associated special subsets $S(t_{j}, 1)$ and $S(t_{j},2)$ of $[n]$. We will set $t_{1} = T-1$, and then inductively set $t_{j} = \sup \{ t : \, t < t_{j-1}, P_{t} \neq P_{t+1} \}$. Finally, note that if $P_{t-1} \neq P_{t}$, the only difference between them is that two elements of $P_{t}$ have been merged into a single element in $P_{t-1}$. Label the set merged at time $t_{j}$ with fewer elements $S(t_{j},1)$, and label the other one $S(t_{j},2)$. If both sets have the same number of elements, set $S(t_{j},1)$ to be the one containing the smallest element (this is, of course, quite arbitrary). \par 
We will be interested in the smallest time $\tau$ such that $P_{T - \tau} = [n]$, a single block (set $\tau = \infty$ if $P_{t}$ is never a single block). Lemma 4.2 of \cite{Smit11}, a small adaptation of classical arguments (see e.g. chapter 7 of \cite{Boll01}), says: 

\begin{lemma} [Connectedness] \label{LemmaConnectednessNarrowMat}
Let $\epsilon > 0$ and assume $n>4$ satisfies $\frac{n}{\log(n)} > \frac{3(1 + 2 \epsilon)(\frac{1}{2} + 2 \epsilon}{\epsilon}$. For the Gibbs sampler on narrow matrices,
\begin{equation*}
P \left[\tau > \left( \frac{1}{2} + 2\epsilon \right) n \log(n) \right] \leq 2n^{-\epsilon}
\end{equation*}
\end{lemma}

The analogous lemma for the other example will be proved in Section \ref{SecCoupForCayley}, Lemma \ref{LemmaConnectednessCayley}. \par 

For both of our walks, we will use two types of coupling, the `proportional' coupling and the `subset' coupling. In both cases, we will set $i(t)^{x} = i(t)^{y}$ and $j(t)^{x} = j(t)^{y}$ at each step. In the proportional coupling, we will also set $\lambda(t)^{x} = \lambda(t)^{y}$. \par 
To discuss the subset coupling, we must define the weight of $X_{t}$ on a subset $S \subset [n]$, which we call $w(X_{t}, S)$. For the simplex walk, we define $w(X_{t},S) = \sum_{s \in S} X_{t}[s]$. For narrow matrices, we define $w(X_{t}, S) = \sum_{s \in S} X_{t}[s,1]$.  The subset couplings associated with subset $S \subset [n]$ is defined immediately prior to Lemma  \ref{LemmaSubsetCoupCayley} in terms of the walk on the simplex (the coupling for the walk on narrow matrices is identical, but uses the representation of that walk in equations \eqref{EqNarrowRepDeltaPos} and \eqref{EqNarrowRepDeltaNeg} rather than the representation for the simplex walk given in equation \eqref{EqCayleyMoveRep}). Roughly, the subset coupling at time $t$ will often set $w(X_{t+1}, S) = w(Y_{t+1}, S)$. We say that a subset coupling of subset $S$ at time $t$ succeeds if that equality holds; otherwise, we say it fails. \par 
In each case, the coupling of $X_{t}$ and $Y_{t}$ during the non-Markovian coupling phase will be as follows. At marked times $t_{j}$, we will perform a subset coupling of $X_{t_{j}}$, $Y_{t_{j}}$ with respect to $S(t_{j}, 1)$. At all other times, we will perform a proportional coupling. This leads to:
\begin{lemma} [Final Coupling] Assume the non-Markovian coupling phase lasts from time $T_{1}$ to $T$, that $P_{T_{1}} = \{ [n] \}$, and that all subset couplings succeed. Then $X_{T} = Y_{T}$.
\end{lemma}
\begin{proof} Let $\mathcal{F}_{t}$ denote the collection of equations $w(X_{t}, S) = w(Y_{t}, S)$ for all $S \in P_{t}$. We will show by induction on $t$ that the equations $\mathcal{F}_{t}$ hold for all $T_{1} \leq t \leq T$.
At time $T_{1}$, we have $w(X_{T_{1}}, [n]) = w(Y_{T_{1}}, [n]) = 1$. By definition of the partition process, if $t$ is not a marked time and all equations $\mathcal{F}_{t}$ hold, then all equations $\mathcal{F}_{t+1}$ also hold. In fact, this is true for any coupling of $\lambda(t)^{x}, \lambda(t)^{y}$ at that step, not just the proportional coupling. \par 
Assume $t = t_{j}$ is a marked time, and that the equations $\mathcal{F}_{t_{j}}$ hold. Then if the equations $\mathcal{F}_{t_{j}+1}$ don't all hold, we must have that either $w(X_{t_{j}+1}, S(t_{j},1)) \neq w(Y_{t_{j}+1}, S(t_{j},1))$ or $w(X_{t_{j}+1}, S(t_{j},2)) \neq w(Y_{t_{j}+1}, S(t_{j},2))$, since none of the terms in the other equations change. By assumption, all subset couplings have succeeded, so $w(X_{t_{j}+1}, S(t_{j},1)) = w(Y_{t_{j}+1}, S(t_{j},1))$. By construction, \\* $w(X_{t_{j}+1}, S(t_{j},2)) = w(X_{t_{j}+1}, S(t_{j},1) \cup S(t_{j},2)) - w(Y_{t_{j}+1}, S(t_{j},1))$ and similarly for $Y_{t_{j}+1}$, so  $w(X_{t_{j}+1}, S(t_{j},2)) = w(Y_{t_{j}+1}, S(t_{j},2))$. Thus, the inductive claim has been proved. \par 
Finally, we note that if  $w(X_{t}, \{ i \}) = w(Y_{t}, \{ i \}) $ for any singleton $\{ i \}$, then $X_{t}[i] = Y_{t}[i]$ for the sampler on the simplex (respectively $X_{t}[i,j] = Y_{t}[i,j]$ for $j \in \{ 1, 2 \}$ for the other sampler). Since $P_{T} = \{ \{ 1 \}, \{2 \}, \ldots, \{ n \} \}$, this proves the lemma. 
\end{proof} \par 
So, in both cases, to show that coupling has succeeded, it is sufficient to show that all subset couplings succeed with high probability.

\section{Contraction for Gibbs Samplers on the Simplex with Geometry}
In this section, we prove a contraction lemma for Gibbs samplers on the simplex associated with a group $G$ and symmetric generating set $R$ of $G$ (that is, $R^{-1} = R$), where $\vert G \vert = n$, $\vert R \vert = m$, and $id$ is the identity element of $G$. We recall briefly some definitions. We write $\Delta_{G} = \{ X \in \mathbb{R}^{G} \vert X[g] \geq 0, \sum_{g \in G} X[g] = 1 \}$. If $X_{t} \in \Delta_{G}$ is a copy of the Markov chain, we take a step by choosing $g \in G$, $r \in R$ and $\lambda \in [0,1]$ uniformly and setting $X_{t+1}[g] = \lambda(X_{t}[g] + X_{t}[gr])$, $X_{t+1}[gr] = (1- \lambda)X_{t}[g] + X_{t}[gr])$, and for all other entries $X_{t+1}[h] = X_{t}[h]$. This walk is closely related to a slow simple random walk on the group. In particular, we let $Z_{t} \in G$ be the random walk that evolves by choosing at each time step a group element $g \in G$ and generator $r \in R$ uniformly at random, and setting $Z_{t+1} = Z_{t}r$ if $Z_{t} = g$, and $Z_{t+1} = Z_{t}$ otherwise. \par
Let $\widehat{K}$ be the transition kernel associated with the random walk $Z_{t}$. Since $R$ is symmetric, the random walk is reversible, so $\widehat{K}$ can be written in a basis of orthogonal eigenvectors with real eigenvalues $1 = \widehat{\lambda_{1}} > \widehat{\lambda_{2}} \geq \ldots \geq \widehat{\lambda_{n}} \geq -1$. Since it is $\frac{1}{2}$-lazy, all eigenvalues are in fact nonnegative. Let $\widehat{\gamma} = 1 - \widehat{\lambda_{2}}$ be the spectral gap of $\widehat{K}$. In this section we will show that

\begin{lemma} [Contraction Estimate for Gibbs Sampler on Cayley Graphs] \label{LemmaContractionCayley}
Let $X_{t}$, $Y_{t}$ be two copies of the Gibbs sampler on the simplex associated with $G$ and $R$, with joint distribution given by a proportional coupling at each step. Then
\begin{equation*}
E[\vert \vert X_{t} - Y_{t} \vert \vert_{2}^{2}] \leq  4 n e^{- \lfloor \frac{t \widehat{\gamma}}{8} \rfloor}
\end{equation*}
\end{lemma}
\begin{proof} We will construct an auxilliary Markov chain on $G$ associated with $X_{t}$, and compare it to the standard random walk $Z_{t}$. Let $X_{t}$, $Y_{t}$ be two copies of the walk, and couple them at each step with the proportional coupling. For $h \in G$, let $S_{t}^{h} = \sum_{g \in G} (X_{t}[g] - Y_{t}[g])(X_{t}[hg] - Y_{t}[hg])$. We will analyze the evolution of the vector $S_{t} = (S_{t}^{id}, \ldots)$. \par
There are three cases to analyze: $h \notin R$ and $h \neq id$, $h \in R$ and $h \neq id$, and $h = id$. Let $\mathcal{F}_{t}$ be the $\sigma$-algebra generated by $X_{s}$ and $Y_{s}$, $0 \leq s \leq t$. For case 1, we have
\begin{align*}
E[& S_{t+1}^{h} \vert \mathcal{F}_{t}] = \left(1 - \frac{4}{n} + \frac{2}{mn} \right) S_{t}^{h}\\
&+ \frac{1}{2mn} \sum_{i \in G} \sum_{r \in R, r \ne h, h^{-1}} [ (X_{t}[i] + X_{t}[ri] - Y_{t}[i] - Y_{t}[ri])(X_{t}[hi] - Y_{t}[hi]) \\
&+ (X_{t}[ri] + X_{t}[i] - Y_{t}[ri] - Y_{t}[i])(X_{t}[hri] - Y_{t}[hri]) \\ &+(X_{t}[h^{-1}i] - Y_{t}[h^{-1}i])(X_{t}[i] + X_{t}[ri] - Y_{t}[i] - Y_{t}[ri]) \\
&+ (X_{t}[h^{-1}ri] - Y_{t}[h^{-1}ri])(X_{t}[i] + X_{t}[ri] - Y_{t}[i] - Y_{t}[ri])] \\
&+  \frac{2}{mn} \sum_{i \in G} [ \frac{1}{6} (X_{t}[i] + X_{t}[hi] - Y_{t}[i] - Y_{t}[hi])^{2} \\
&+ (X_{t}[i] - Y_{t}[i])(X_{t}[hi] - Y_{t}[hi]) + (X_{t}[i] - Y_{t}[i])(X_{t}[h^{2}i] - Y_{t}[h^{2}i])] \\
&= \left( 1 - \frac{2}{n} + \frac{2}{3mn} \right)S_{t}^{h} + \frac{2}{3mn} S_{t}^{id} + \frac{2}{mn} S_{t}^{h^{2}} \\
& + \frac{1}{2mn} \sum_{r \in R, r \ne h, h^{-1}} (S_{t}^{hr^{-1}} + S_{t}^{hr} + S_{t}^{rh} + S_{t}^{rh^{-1}})
\end{align*} \par 
and we note that the sum of the coefficients is $1 - \frac{2}{3mn}$. For case 2, we have
\begin{align*}
E[& S_{t+1}^{h} \vert \mathcal{F}_{t}] = \left( 1 - \frac{4}{n} \right) S_{t}^{h} + \frac{2m}{mn} S_{t}^{h} \\
&+ \frac{1}{2mn} \sum_{r \in R} (S_{t}^{hr^{-1}} + S_{t}^{hr} + S_{t}^{rh} + S_{t}^{rh^{-1}}) \\
&=  \left( 1 - \frac{2}{n} \right) S_{t}^{h} + \frac{1}{2mn} \sum_{r \in R} (S_{t}^{hr^{-1}} + S_{t}^{hr} + S_{t}^{rh} + S_{t}^{rh^{-1}})
\end{align*} \par 
where the sum of the coefficients is 1. Finally, in case 3, we have
\begin{align*}
E[&S_{t+1}^{id} \vert \mathcal{F}_{t}] = \left( 1 - \frac{2}{n} \right) S_{t}^{id} + \frac{2}{3mn} \sum_{r \in R} \sum_{i \in G} (X_{t}[i] + X_{t}[ri] - Y_{t}[i] - Y_{t}[ri])^{2} \\
&= \left( 1 - \frac{2}{3n} \right)S_{t}^{id} + \frac{4}{3mn} \sum_{r \in R} S_{t}^{r}
\end{align*} 
and here the sum of the coefficients is $1 + \frac{2}{3n}$. If we rewrite $U_{t}^{id} = \frac{1}{2} S_{t}^{id}$, and $U_{t}^{g} = S_{t}^{g}$ for $g \neq id$, and $U_{t} = (U_{t}^{id}, \ldots)$, then we find the following transformations. For case 1, we have
\begin{align} \label{EqNewChain1}
 E[& U_{t+1}^{h} \vert \mathcal{F}_{t}] =  \left( 1 - \frac{2}{n} + \frac{2}{3mn} \right)U_{t}^{h} + \frac{4}{3mn} U_{t}^{id} + \frac{2}{mn} U_{t}^{h^{2}} \\
& + \frac{1}{2mn} \sum_{r \in R, r \ne h, h^{-1}} (U_{t}^{hr^{-1}} + U_{t}^{hr} + U_{t}^{rh} + U_{t}^{rh^{-1}}) \nonumber
\end{align} \par 
For case 2, we have
\begin{align} \label{EqNewChain2}
E[& U_{t+1}^{h} \vert \mathcal{F}_{t}] =  \left( 1 - \frac{2}{n} \right) U_{t}^{h} + \frac{1}{2mn} \sum_{r \in R} (U_{t}^{hr^{-1}} + U_{t}^{hr} + U_{t}^{rh} + U_{t}^{rh^{-1}})
\end{align}
Finally, in case 3, we have
\begin{align} \label{EqNewChain3}
E[&U_{t+1}^{id} \vert \mathcal{F}_{t}] =  \left( 1 - \frac{2}{3n} \right)U_{t}^{id} + \frac{2}{3mn} \sum_{r \in R} U_{t}^{r}
\end{align} 
where the sum of the coefficients is now 1 in all three cases. Since the coefficients are now all nonnegative numbers summing to 1, the equations \eqref{EqNewChain1} to \eqref{EqNewChain3} define the transition kernel of a Markov chain on $G$. From equation \eqref{EqNewChain3}, this chain sends the identity to itself with probability $1 - \frac{2}{3n}$, and to a uniformly chosen element of $R$ with the remaining probability; Equations \eqref{EqNewChain1} and \eqref{EqNewChain2} describe transitions from $h \in R$ and $h \neq id, h \notin R$ respectively. Call the transition kernel $K$. \par 
Before analyzing the chain, we note that $\sum_{i \in G} (X_{t}[i] - Y_{t}[i]) = 0$, and so
\begin{align*}
0 &= \left( \sum_{i \in G} (X_{t}[i] - Y_{t}[i]) \right)^{2} \\
&= \sum_{i \in G} (X_{t}[i] - Y_{t}[i])^{2} + \sum_{i \ne j} (X_{t}[i] - Y_{t}[i])(X_{t}[j] - Y_{t}[j]) \\
&= S_{t}^{id} + \sum_{h \ne id} S_{t}^{h}
\end{align*} \par 
From this calculation, if $\langle v, (2,1,1,\ldots,1) \rangle = 0$, then $\langle Kv, (2,1,1, \ldots, 1) \rangle = 0$ as well. By direct computation, $\pi = \frac{1}{n+1} (2, 1,1, \ldots, 1)$ is a reversible measure for $K$. It is also clear that the distribution $\widehat{\pi} = \frac{1}{n} (1,1,\ldots,1)$ is the reversible measure for $\widehat{K}$.\par 
We are now ready to compare the chains. Recall from \cite{DGJM06} that the Dirichlet form associated to a Markov chain with transition kernel $Q$ and stationary distribution $\nu$ is given by
\begin{equation*}
\mathcal{E} (\phi) = \frac{1}{2} \sum_{h,g \in G} \nu(g) Q(g,h) (\phi(g) - \phi(h))^{2}
\end{equation*}
Let $\mathcal{E}$ and $\widehat{\mathcal{E}}$ be the Dirichlet forms associated with $K$ and $\widehat{K}$ respectively. Then by comparing terms, it is clear that $\mathcal{E} (\phi) \geq \frac{1}{4} \widehat{\mathcal{E}}(\phi)$ for any $\phi$ and $\frac{\pi}{\widehat{\pi}}$, $\frac{\widehat{\pi}}{\pi} \leq 2$. By e.g. Lemma 13.12 of \cite{LPW09}, this implies $\gamma \geq \frac{1}{8} \widehat{\gamma}$. \par 
Recall that if $\langle v, \pi \rangle = 0$, then $\langle K^{m} v, \pi \rangle = 0$ as well. In particular $K$ applied to the subspace orthogonal to $\pi$ has $L^{2} \rightarrow L^{2}$ operator norm at most $1 - \gamma$. Thus, we have for any $v$ in that subspace
\begin{equation*}
\vert \vert K^{m} v \vert \vert_{2} \leq e^{- \lfloor \gamma m \rfloor} \vert \vert v \vert \vert_{2}
\end{equation*}
going back to our original situation, we are interested in the vector $(S_{t}^{g})$. At time 0, $S_{0}^{id} \leq 2$, and by Cauchy-Schwarz $\vert S_{0}^{h} \vert \leq S_{0}^{id}$. \par 

Thus, $\vert \vert U_{0} \vert \vert_{1} \leq 2n$, and of course $\vert S_{t}^{id} \vert \leq \vert \vert S_{t} \vert \vert_{1} \leq 2 \vert \vert U_{t} \vert \vert_{1}$. So we find that
\begin{equation*}
E[\vert S_{t}^{id} \vert] \leq 4n e^{- \lfloor \frac{ t \widehat{\gamma} }{8} \rfloor}
\end{equation*}
which is the contraction estimate in Lemma \ref{LemmaContractionCayley}. \end{proof}

\section{Coupling for Gibbs Samplers on the Simplex with Geometry} \label{SecCoupForCayley}
Having shown contraction, we must now show convergence in total variation distance. First, the analogue to Lemma \ref{LemmaConnectednessNarrowMat}:
\begin{lemma} [Connectedness for Gibbs Sampler on Cayley Graphs] \label{LemmaConnectednessCayley}
Let $\tau$ be as defined immediately before Lemma \ref{LemmaConnectednessNarrowMat} and let $\widehat{\gamma}$ be as defined immediately before Theorem \ref{ThmConvRateCayley}. Then for $t > 8 \frac{(C+3) \log(n)}{\widehat{\gamma}}$, we have
\begin{equation*}
P[\tau > t] \leq 2n^{-C}
\end{equation*}
\end{lemma} 
\begin{proof}
We consider a graph-valued process $G_{t}$, where $G_{0}$ is a graph with no edges, and vertex set equal to the group $G$. To construct $G_{t+1}$ from $G_{t}$, choose elements $g \in G$ and $r \in R$ uniformly at random, and add the edge $(g, gr)$ if it isn't already in $G_{t}$. We note that $\tau > t$ if and only if $G_{t}$ is not connected, so we would like to estimate the time at which $G_{t}$ becomes connected. \par 
First, fix two elements $x, y \in G$. We'd like to see if $x$, $y$ are in the same component of $G_{t}$. To do so, let $X_{t}$, $Y_{t}$ be two copies of the Gibbs sampler described in the last section, with $X_{0} = x$, $Y_{0} = y$. Couple $X_{t}$, $Y_{t}$ and $G_{t}$ by the proportional coupling. Then assume $x$, $y$ are in different components $C_{x}$, $C_{y}$ at time $t$. We would have
\begin{align*}
\sum_{g} \vert X_{t}[g] - Y_{t}[g] \vert^{2} &\geq \sum_{g \in C_{x}} \frac{1}{\vert C_{x}\vert^{2}} +  \sum_{g \in C_{y}} \frac{1}{\vert C_{y}\vert^{2}} \\
&\geq \frac{4}{n}
\end{align*}
By Markov's inequality, then,
\begin{equation*}
P[C_{x} \neq C_{y}] \leq \frac{n}{4} E[\sum_{g} \vert X_{t} - Y_{t} \vert^{2}] 
\end{equation*}
and so, by standard union bound for fixed $x$ over all $y$, if $A_{t}$ is the event that $G_{t}$ is disconnected,
\begin{align*}
P[A_{t}] &\leq \frac{n^{2}}{4} \sup_{\mu, \nu} E[\sum_{g} \vert X_{t} - Y_{t} \vert^{2} \vert X_{0} = \mu, Y_{0} = \nu] \\
&\leq n^{3} e^{- \lfloor\frac{ t \widehat{\gamma}}{8} \rfloor}
\end{align*}
where the last inequality is due to Lemma \ref{LemmaContractionCayley}.
\end{proof} \par 
Next, we define subset couplings and discuss success probabilities for this walk. Fix points $X_{t}$, $Y_{t}$, subset $S \subset[n]$ and updated coordinates $i = i(t) \in S$, $j = j(t) \notin S$. The next step is to construct a pair of uniform random variables $\lambda_{x} = \lambda(t)^{x}$ and $\lambda_{y} = \lambda(t)^{y}$ with which to update the chains $X_{t}$ and $Y_{t}$ respectively. Assume first that $\frac{X_{t}[i] + X_{t}[j]}{Y_{t}[i] + Y_{t}[j]} < 1$, and choose $\lambda_{y}$ uniformly in $[0,1]$. Then set
\begin{align} \label{CgEqSubsetCoupling}
\lambda_{x} = \lambda_{y} \frac{Y_{t}[i] + Y_{t}[j]}{X_{t}[i] + X_{t}[j]} + \frac{1}{X_{t}[i] + X_{t}[j]} \sum_{s \in S / \{ i \}} (Y_{t}[s] - X_{t}[s])
\end{align} 
if that results in a value between 0 and 1. Otherwise, choose $\lambda_{x}$ independently of $\lambda_{y}$, according to the density:
\begin{align} \label{CgEqSubsetRemainder}
f(\lambda) = C \left( 1 - \frac{X_{t}[i] + X_{t}[j]}{Y_{t}[i] + Y_{t}[j]} \textbf{1}_{ g(\lambda) \in [0,1]}(\lambda) \right)
\end{align}
where $C$ is a normalizing constant, and 
\begin{equation*}
g(\lambda) = \lambda \frac{Y_{t}[i] + Y_{t}[j]}{X_{t}[i] + X_{t}[j]} + \frac{1}{X_{t}[i] + X_{t}[j]} \sum_{s \in S / \{ i \}} (Y_{t}[s] - X_{t}[s])
\end{equation*} \par 
From the assumption that $\frac{X_{t}[i] + X_{t}[j]}{Y_{t}[i] + Y_{t}[j]} < 1$, it is easy to see that $f$ really is a density on $[0,1]$. From its construction as a remainder density, it is easy to check that under this coupling, $\lambda_{x}$ is uniformly distributed on $[0,1]$. If $\frac{X_{t}[i] + X_{t}[j]}{Y_{t}[i] + Y_{t}[j]} > 1$, an analogous construction will work. More precisely, in this case choose $\lambda_{x}$ first, and then choose $\lambda_{y}$ to satisfy equation \ref{CgEqSubsetCoupling} if the result is in $[0,1]$, rather than choosing $\lambda_{y}$ first. If the result is not in $[0,1]$, then choose $\lambda_{y}$ according to its remainder measure, given by  equation \eqref{CgEqSubsetRemainder} with $X_{t}$ and $Y_{t}$ switched and $g$ replaced by $g^{-1}$. Note that if equation \ref{CgEqSubsetCoupling} is satisfied, then $w(S, X_{t+1}) = w(S, Y_{t+1})$. \par 

For a pair of points $(x,y)$ in the simplex, a pair of update entries $(i,j)$, and a subset $S \subset [n]$ of interest such that $i \in S$ and $j$ not in $S$, we define $p(x,y,i,j,S)$ to be the probability that the associated subset coupling succeeds. Then Lemma 4.4 from \cite{Smit11} gives a lower bound on this probability: \\

\begin{lemma} [Subset Coupling] \label{LemmaSubsetCoupCayley}

Assume $n \geq 6$, and fix $0 \leq b \leq f-1$. For a pair of vectors $(x,y)$ satisfying $\sup_{i} \vert x_{i} - y_{i} \vert \leq n^{-f}$ and $\inf_{i} x_{i}, \inf_{i} y_{i} \geq n^{-b}$, we have $p(x,y,i,j,S) \geq 1 - 3n^{b + 1 - f}$ uniformly in $S$ and possible $i,j$.
\end{lemma}

In general, it is possible to choose $x,y,i,j,S$ so that the probability of success is 0 under any coupling, and the lemma is quite restrictive. Having bounded the probability of failure when $X_{t}$, $Y_{t}$ are close, we must show that they remain close with high probability. Define for $v \in \mathbb{R}^{n}$ and $S \subset [n]$ the quantity $\vert \vert v \vert \vert_{S} = \sum_{s \in S} \vert v[s] \vert$. We will need Lemma 4.5 from \cite{Smit11}:

\begin{lemma}[Closeness] \label{LemmaSmallness}
Let $X_{t}$, $Y_{t}$ be either the of the chains described in this paper, coupled as described in Section \ref{SecPartProc}, and assume that $P_{T_{1}} = \{ [n] \}$, that all subset couplings up to time $s-1$ have succeeded, and that $\vert \vert X_{T_{1}} - Y_{T_{1}} \vert \vert_{1} < \epsilon$. Then $\vert \vert X_{s} - Y_{s} \vert \vert_{S} < \epsilon$ for every $S$ in $P_{t}$.
\end{lemma}

Related to this, Lemma 4.6 from \cite{Smit11} shows that $X_{t}$, $Y_{t}$ rarely have entries close to 0: \\
\begin{lemma}[Largeness] \label{LemmaLargeness}
$P[\inf_{1 \leq i \leq n} \inf_{0 \leq t \leq T-1} Y_{t}[i] \leq T n^{-2.5 -k}] \leq 2 T n^{-k}$ for $n > \max(2k, 4096)$.  \end{lemma}

This lets us complete the calculation. We split the run of $\frac{8 C \log(n)}{\widehat{\gamma}}$ steps into an initial contractive phase of length $T_{1} = 8(\frac{5}{6}C + \frac{7}{2}) \frac{\log(n)}{\widehat{\gamma}}$ and a second, coupling, phase of length $T_{2} = 8(\frac{1}{6}C - \frac{7}{2}) \frac{\log(n)}{\widehat{\gamma}}$. By Lemma 5, $X_{T} = Y_{T}$ unless one of the subset couplings fails or  $\tau > T_{2}$. By Lemma 7, $P[\tau > T_{2}] \leq 2n^{-\frac{1}{6} C + 3.5}$. Thus, it remains only to bound the probability that a subset coupling fails, assuming $\tau \leq T_{2}$. \par 

By Lemma 6, Markov's inequality and the bound $\vert \vert v \vert \vert_{1} \leq \sqrt{2n} \vert \vert v \vert \vert_{2}$, $P[\vert \vert X_{T_{1}} - Y_{T_{1}} \vert \vert_{1} \geq n^{ \frac{C}{3} + \frac{1}{2}}] \leq 4 \sqrt{2} n^{- \frac{C}{6} - 1}$. It is easy to show that $\widehat{\gamma} \geq \frac{1}{2n^{4}}$ for simple random walk on any Cayley graph (e.g. by naive bounds with Theorem 13.14 of \cite{LPW09}). Combining this with Lemma 10, we have $P[\inf_{g \in G, T_{1} \leq t \leq T_{2}} Y_{t}[g] \leq n^{-\frac{C}{6} - 3}] \leq 2 n^{4.5 - \frac{C}{6}}$ for $n$ satisfying $n > \max(4096, 2b, 2 \left( \frac{C}{6} - 3.5 \right) \log(n))$. By Lemmas 8 and 9, as long as $\inf_{g \in G, T_{1} \leq t \leq T_{2}} Y_{t}[g] > n^{-b}$ and $\vert \vert X_{T_{1}} - Y_{T_{1}} \vert \vert_{1} < n^{ \frac{C}{3} + \frac{1}{2}}$ and $C > \frac{103}{4}$, we conclude that the all subset couplings succeed with probability at least $1 - 3 n^{4.5 - \frac{C}{6}}$. Combining the bounds in this paragraph, as long as $\tau \leq T_{2}$, all subset couplings succeed with probability at least $1 - 6 n^{4.5 -\frac{C}{6}}$. \par 

Putting together the bounds in the last line of each of the preceeding two paragraphs, $P[Y_{T} \neq X_{T}] \leq 7 n^{4.5 - \frac{C}{6}}$, which proves the upper bound in the theorem.

\section{Lower Bounds for  Gibbs Samplers on the Simplex with Geometry}
In this section, we prove lower bounds on the mixing time of the Gibbs sampler on the simplex. The results are similar to those of \cite{RaWi05a}, though the method is different and elementary. Begin by calculating
\begin{align} \label{EqLowBdCalc}
E[X_{t+1}[g] \vert X_{t}] &= \left( 1 - \frac{2}{n} \right) X_{t}[g] + \frac{2}{n} \frac{1}{m} \sum_{r \in R} \frac{1}{2} (X_{t}[g] + X_{t}[gr]) \\
&= \left( 1 - \frac{1}{n} \right) X_{t}[g] + \frac{1}{n} \frac{1}{m} \sum_{r \in R} X_{t}[gr]
\end{align} \par 
In particular, let $K$ be the transition matrix on $G$ given by $K[g,g] = 1 - \frac{1}{n}$, and $K[g,gr] = \frac{1}{n m} $ for $r \in R$. This is the standard `edge'-based random walk on $G$ with generating set $R$ described above. By equation \eqref{EqLowBdCalc}, $E[X_{t}] = K^{t} X_{0}$. Note that this is not the same $K$ as was used earlier in the section while proving the upper bound on the mixing time. By the earlier assumptions on $R$, $K$ is reversible with respect to the uniform measure on $G$. Furthermore, it is orthogonally diagonalizable with real eigenvalues $1 = \beta_{1} > \beta_{2} \geq \ldots \geq \beta_{n} \geq 0$.\par 
Next, let $v$ be an eigenvector of $K$ with eigenvalue $\beta_{2}$, normalized so that $\vert \vert v \vert \vert_{2} = 1$ and $\vert \vert v - Kv \vert \vert_{2} = \gamma$, the spectral gap of $K$. Let $\Pi$ be the collection of vectors with nonnegative entries summing to 1, and let $w \in \Pi$ maximize the inner product $\langle v,w \rangle$ among such vectors; such a vector exists by the compactness of $\Pi$. Let $X_{t}$ be a copy of the Markov chain begun from $X_{0} \stackrel{D}{=} w$, then  $E[\langle X_{t},v \rangle] = (1 - \gamma)^{t} \langle w,v \rangle$. On the other hand, if $A_{t,d}$ is the event that $\langle X_{t},v \rangle > d$, 
\begin{align*}
E[\langle X_{t},v \rangle] &= E[ \langle X_{t},v \rangle 1_{A_{t,d}}] + E[\langle X_{t},v \rangle 1_{A_{t,d}^{c}}] \\
&\leq \langle X_{0}, v \rangle P[A_{t,d}] + d \\
\end{align*}
where the second inequality takes advantage of the maximality of $X_{0}$. Thus,
\begin{equation*}
P[A_{t,d}] \geq \frac{E[\langle X_{t},v \rangle] - d}{\langle X_{0}, v \rangle}
\end{equation*}
putting the two inequalities together,
\begin{equation} \label{IneqBadSetBasic}
P[A_{t,d}] \geq (1 - \gamma)^{t} - \frac{d}{\langle X_{0},v \rangle}
\end{equation} \par 
The next step is to prove that $\langle X_{0}, v \rangle \geq \frac{1}{2\sqrt{n}}$. Let $P \subset [n]$ be the collection of indices so that $v[p] \geq 0$ for $p \in P$. Without loss of generality, assume $\sum_{p \in P} v_{p}^{2} > \frac{1}{2}$. Now set $\lambda^{-1}  = \sum_{p \in P} v_{p} \leq \sqrt{n}$. Then consider the distribution given by $\mu_{p} = \lambda v_{p}$ for $p \in P$, and $\mu_{p} = 0$ for $p \notin P$: 
\begin{align*}
\langle \mu, v \rangle &= \lambda \sum_{p \in P} v_{p}^{2} \\
& \geq \frac{1}{2\sqrt{n}}
\end{align*}
and so by inequality \eqref{IneqBadSetBasic},
\begin{equation*}
P[A_{t,d}] \geq (1 - \gamma)^{t} - 2 d \sqrt{n} 
\end{equation*}

Now, let $Y \in \Delta_{G}$ be chosen according to the uniform distribution. Then $E[ \langle Y,v \rangle ] = 0$, and
\begin{align*}
E[\langle Y,v \rangle^{2}] &= E[(\sum_{i \in G} Y[i] v_{i})^{2}] \\
&= \sum_{i \in G} v_{i}^{2} E[Y[i]^{2}] +  \sum_{i \neq j \in G} v_{i} v_{j} E[Y[i] Y[j]] \\
&\leq \frac{2}{n^{2}} + 0 \\
\end{align*} \par 
So, by Chebyshev's inequality, $P[\langle Y,v \rangle > d] \leq \frac{2}{d^{2} n^{2}}$. Putting this together with the inequality above, letting $d = n^{-\frac{5}{6}}$ and defining $P[A_{\infty,d}] = \lim_{t \rightarrow \infty} P[A_{t,d}]$,
\begin{equation*}
P[A_{t,d}] - P[A_{\infty,d}] \geq \gamma^{t} - 4 n^{-\frac{1}{3}}
\end{equation*}
And the lower bound follows immediately. 

\section{Contraction and Narrow Matrices}
We begin with some quick observations about the geometry of our space. It is the part of an $(n-1)$-dimensional affine subspace of $\mathbb{R}^{2n}$ that lies in the upper orthant. Our updates are in fact moves along 1-dimensional pieces of this subspace, even though we are updating four entries. While the original motivation for this sampler comes from statistics (see e.g. \cite{DiEf85}), it is being treated here primarily as an example of a chain that is somewhere between the standard Gibbs sampler on the simplex and an analogous Gibbs sampler on doubly-stochastic matrices or Kac's famous walk on the orthogonal group. The former was analyzed by the author in \cite{Smit11}. Matching bounds on Total variation mixing time are not known for either the Gibbs sampler on doubly-stochastic matrices or Kac's walk. The best such bounds to date can be found in \cite{Smit12} and \cite{Jian12} respectively. Both mixing bounds are polynomials with small but probably incorrect degrees, and both are based on much more complicated non-Markovian coupling arguments. \par 

In this section, we will prove the following contractivity estimate for the Gibbs sampler on narrow matrices. The original proof was a direct translation of the path-coupling argument for $k$ by $n$ matrices found in \cite{Smit12}. The following greatly simplified proof was suggested by an extremely helpful reviewer.

\begin{lemma} [Weak Convergence on Narrow Matrices] \label{LemmaWeakConMat}
If $X_{t}$ and $Y_{t}$ are coupled under the proportional coupling for time $0 \leq t \leq T_{1} = 3 k n \log(n)$, then
\begin{equation*}
P[\vert \vert X_{T_{1}} - Y_{T_{1}} \vert \vert_{1} \geq \epsilon ] \leq  \epsilon^{-1} n^{-k}
\end{equation*}
\end{lemma}

We make some basic remarks about the chain, beginning with an alternative description of the transition probabilities. Define $\delta_{t}[i,j] = 2 - X_{t}[i,1] - X_{t}[j,1]$, and $\epsilon_{t}[i,j] = 2 - Y_{t}[i,1] - Y_{t}[j,1]$. Then a step of the chain can be defined in the following way. Choose $i,j$ as before, and choose $\lambda \stackrel{D}{=} U[0,1]$. If $\delta_{t}[i,j] \geq 0$, then we update according to:
\begin{align} \label{EqNarrowRepDeltaPos}
X_{t+1}[i,1] &= \lambda (2 - \delta_{t}[i,j]) \nonumber \\
X_{t+1}[j,1] &= (1 - \lambda) (2 - \delta_{t}[i,j]) \\
X_{t+1}[i,2] &= 2(1 - \lambda) + \lambda \delta_{t}[i,j] \nonumber \\
X_{t+1}[j,2] &= 2 \lambda + (1-\lambda) \delta_{t}[i,j] \nonumber 
\end{align}
If $\delta_{t}[i,j] < 0$, we update according to:
\begin{align} \label{EqNarrowRepDeltaNeg}
X_{t+1}[i,1] &= 2 \lambda - (1 - \lambda) \delta_{t}[i,j] \nonumber \\
X_{t+1}[j,1] &= 2(1 - \lambda) - \lambda \delta_{t}[i,j] \\
X_{t+1}[i,2] &= (1 - \lambda) (2 + \delta_{t}[i,j]) \nonumber \\
X_{t+1}[j,2] &= \lambda (2 + \delta_{t}[i,j]) \nonumber 
\end{align} \par 
Note that in both cases, a larger value of $\lambda$ means a larger value of $X_{t+1}[i,1]$. We are now ready to describe the proportional coupling: as in the simplex case, we choose the same value of $\lambda$ for both chains in the above representation. This leads to the following contraction estimate:

\begin{lemma}[$L^{2}$ Contractivity] \label{LemmaL2Cont}If $X_{t}$, $Y_{t}$ are coupled under the proportional coupling for time $0 \leq t \leq T_{1}$, then
\begin{equation*}
E[||X_{T_{1}} - Y_{T_{1}}||_{2}^{2}] \leq (1 - \frac{2}{3n})^{T_{1}} ||X_{0} - Y_{0}||_{2}^{2}
\end{equation*}
\end{lemma}
\begin{proof} 
We begin the proof by calculating the change in the $L^{2}$ norm during a single move. Let $F_{t}[i,j]$ be the event that coordinates $i,j$ are updated at time $t$. If $\delta_{t}[i,j] \epsilon_{t}[i,j] \geq 0$, we find:
\begin{align*}
\Delta_{t} &\equiv E[(X_{t+1}[i,1] - Y_{t+1}[i,1])^{2} + (X_{t+1}[j,1] - Y_{t+1}[j,1])^{2} \\
&+ (X_{t+1}[i,2] - Y_{t+1}[i,2])^{2} + (X_{t+1}[j,2] - Y_{t+1}[j,2])^{2} \vert F_{t}[i,j]] \\
&=2 E[ (\lambda(2 - \delta_{t}[i,j]) - \lambda(2 - \epsilon_{t}[i,j]))^{2} + (\lambda \delta_{t}[i,j] - \lambda \epsilon_{t}[i,j])^{2} \vert F_{t}[i,j]] \\
&= \frac{4}{3} (\delta_{t}[i,j] - \epsilon_{t}[i,j])^{2}
\end{align*} \par 

If $\delta_{t}[i,j] \epsilon_{t}[i,j] < 0$, we find:

\begin{align*}
\Delta_{t} &= E[(X_{t+1}[i,1] - Y_{t+1}[i,1])^{2} + (X_{t+1}[j,1] - Y_{t+1}[j,1])^{2} \\
&+ (X_{t+1}[i,2] - Y_{t+1}[i,2])^{2} + (X_{t+1}[j,2] - Y_{t+1}[j,2])^{2} \vert F_{t}[i,j]] \\
&= 2 E[ (\lambda(2 - \delta_{t}[i,j]) - 2 \lambda + (1 - \lambda)\epsilon_{t}[i,j])^{2} \vert F_{t}[i,j] ] \\
&+ E[((1 - \lambda)(2 - \delta_{t}) - 2 (1 - \lambda) + \lambda \epsilon_{t}[i,j])^{2} \vert F_{t}[i,j]] \\
&= \frac{4}{3}(\delta_{t}[i,j] - \epsilon_{t}[i,j])^{2} + \frac{1}{3} \delta_{t}[i,j] \epsilon_{t}[i,j] \\
&< \frac{4}{3}(\delta_{t}[i,j] - \epsilon_{t}[i,j])^{2}
\end{align*}

As in the simplex case, we can calculate the sums of terms like $(\epsilon_{t}[i,j] - \delta_{t}[i,j])^{2}$ in terms of sums of terms like $(X_{t}[i,j] - Y_{t}[i,j])^{2}$. We first note that
\begin{align*}
0 &= (\sum_{i=1}^{n} X_{t}[i,1] - Y_{t}[i,1])^{2} \\
&= \sum_{i=1}^{n}  (X_{t}[i,1] - Y_{t}[i,1])^{2} + 2 \sum_{i < j} (X_{t}[i,1] - Y_{t}[i,1])(X_{t}[j,1] - Y_{t}[j,1])
\end{align*} \par 
If $\delta_{t}[i,j] \epsilon_{t}[i,j] \geq 0$ for all $i,j$, we can write the first line of the following computation:
\begin{align}  \label{EqDeltaMinusEpsilon}
\sum_{i \neq j} (\delta_{t}[i,j] - \epsilon_{t}[i,j])^{2} &= \sum_{i \neq j}(X_{t}[i,1] + X_{t}[j,1] - Y_{t}[i,1] - Y_{t}[j,1])^{2} \nonumber \\
&= \sum_{i \neq j} [(X_{t}[i,1] - Y_{t}[i,1])^{2} + (X_{t}[j,1] -Y_{t}[j,1])^{2} \nonumber \\
&+ 2 (X_{t}[i,1] -Y_{t}[i,1])(X_{t}[j,1] -Y_{t}[j,1])] \nonumber \\
&= (n-2) \sum_{j=1}^{2} \sum_{i=1}^{n} (X_{t}[i,j] - Y_{t}[i,j])^{2}
\end{align} 
Then the final contraction is given by:
\begin{align*}
\Delta_{t} &= E[||X_{t+1} - Y_{t+1} ||_{2}^{2} | X_{t}, Y_{t}] \\
&= \frac{1}{n(n-1)} \sum_{k=1}^{2} \sum_{m=1}^{n} \sum_{i \neq j} E[(X_{t+1}[m,k] - Y_{t+1}[m,k])^{2} | F_{t}[i,j]] \\
&= \frac{1}{n(n-1)} \sum_{k=1}^{2} \sum_{m=1}^{n} ( \sum_{i,j \neq m} (X_{t}[m,k] - Y_{t}[m,k])^{2} \\
&+ 2 \sum_{j \neq m} E[(X_{t+1}[m,k] - Y_{t+1}[m,k])^{2} | F_{t}[i,j]]) \\
&= (1 - \frac{2}{n}) ||X_{t} - Y_{t} ||_{2}^{2} + \frac{1}{n(n-1)} \sum_{i \neq j} \frac{4}{3} (\delta_{t}[i,j] - \epsilon_{t}[i,j])^{2} \\
&\leq (1 - \frac{2}{3n}) ||X_{t} - Y_{t}||_{2}^{2}
\end{align*}
where the last line is due to equation \eqref{EqDeltaMinusEpsilon}. 
\end{proof} \par 

Lemma \ref{LemmaWeakConMat} follows from this bound, the inequality $\vert X_{t}[i] - Y_{t}[i] \vert \leq  \vert \vert X_{t} - Y_{t} \vert \vert_{2}$, and Markov's inequality.

\section{Coupling for Narrow Matrices}
In this section, we show that subset couplings are likely to succeed, and finish the proof of Theorem 2. The main lemma is:

\begin{lemma}[Coupling for Nearby Points] \label{LemmaNearbyCoupling}
Fix $a > 7.5$ and $n > \max(4096, a+3.5)$ satisfying $\frac{n}{\log(n)} > \frac{3(1+2c)(\frac{1}{2} + 2c)}{c}$. Let $Y_{t}$ be a copy of the chain started at stationarity and assume that $X_{t}$ is a copy of the chain which satisfies $\vert \vert X_{T_{1}} - Y_{T_{1}} \vert \vert_{1} \leq n^{-a}$. Then for $T_{2} > (\frac{1}{2} + 2c) n \log(n)$ and $T = T_{1} + T_{2}$ we have $P[X_{T} \neq Y_{T}] \leq 10 n^{3.75 - \frac{a}{2}} + n^{-c}$  \end{lemma}

Construct a partition process from time $T_{1}$ to time $T = T_{1} + (\frac{1}{2} + 2c) n \log(n)$. Our first step is to define subset couplings and show that if $X_{t}$ and $Y_{t}$ are very close to each other and not too close to certain hyperplanes, then any subset couplings are likely to succeed. To define a subset coupling of $X_{t}$ and $Y_{t}$, fix the subset $S$ of interest and common update variables $i = i(t) \in S$ and $j = j(t) \notin S$. If $\delta_{t}[i,j] \epsilon_{t}[i,j] \geq 0$, then the coupling of $\lambda_{t}^{x}$ and $\lambda_{t}^{y}$ is exactly as described for the other walk immediately before Lemma \ref{LemmaSubsetCoupCayley}. Otherwise, assume $\delta_{t}[i,j] < 0$ and $\epsilon_{t}[i,j] > 0$, and choose $\lambda_{t}^{x}$ from $[0,1]$ uniformly at random. Then set $\lambda_{t}^{y}$ to be the number which satisfies $w(X_{t+1}, S) = w(Y_{t+1},S)$ if such a number exists and is in the interval $[0,1]$. Just as with equation \eqref{CgEqSubsetCoupling}, the  measure (with mass less than 1) on $\lambda_{t}^{y}$ that this assignment defines minorizes the uniform distribution, and so leaves a remainder distribution analogous to that given in equation \eqref{CgEqSubsetRemainder}. If there is no value of $\lambda_{t}^{y}$ in $[0,1]$ which would allow $w(X_{t+1}, S) = w(Y_{t+1},S)$, then choose $\lambda_{t}^{y}$ uniformly from this remainder distribution.  If $\delta_{y}[i,j] > 0$ and $\epsilon_{t}[i,j] < 0$, the same construction works, but with $\lambda_{t}^{y}$ chosen first, and $\lambda_{t}^{x}$ chosen to satisfy $w(X_{t+1}, S) = w(Y_{t+1},S)$. \par 

The following lemma shows that, after a moderate number of steps, $X_{t}$ and $Y_{t}$ are unlikely to be too close to the boundary of our convex set:

\begin{lemma}[Largeness] \label{LemmaLargenessMatrix} For  $n > \max(2k, 4096)$, 
\begin{equation*}
P[\inf_{i,j} \inf_{T_{1} \leq t \leq T_{1} + n^{2} - 1} Y_{t}[i,j] \leq n^{-5.5 - k}], P[\sup_{i,j} \sup_{T_{1} \leq t \leq T_{1} + n^{2} - 1}Y_{t}[i,j] \geq 2 - n^{-5.5-k}] \leq 2n^{-k}
\end{equation*}
\end{lemma}

\begin{proof}
Our proof will be via comparison to a Gibbs sampler on the simplex, studied by the author in \cite{Smit11}. Let $X_{t}$ be a copy of our Gibbs sampler on $2$ by $n$ matrices, and let $S_{t}$ be a Gibbs sampler on the simplex $\Delta_{n} = \{ S \in \mathbb{R}^{n} | S[i] \geq 0, \sum_{i=1}^{n} S[i] = 1 \}$. To make a move in this Gibbs sampler, choose distinct coordinates $1 \leq i,j \leq n$ and $0 \leq \lambda \leq 1$ uniformly at random, and update entry $S_{t}[i]$ to $\lambda (S_{t}[i] + S_{t}[j])$ and entry $S_{t}[j]$ to $(1 - \lambda)(S_{t}[i] + S_{t}[j])$, keeping all other entries fixed. This is identical to the other sampler in this note, with generating set $R = G \backslash \{ id \}$. \par 
Since $\sum_{i} S_{0}[i] = 1$, for any given $X_{0}$ it is possible to choose a corresponding $S_{0}$ such that $X_{0}[i,1] \geq S_{0} [i]$ for all $i$, without the row sum condition interfering. Next, under our descriptions there is a natural proportional coupling of $X_{t}$ and $S_{t}$, given by always choosing $i,j$ and $\lambda$ to be the same. We claim that under this coupling, $X_{t}[i,1] \geq S_{t}[i]$ for all $t > 0$ and all $1 \leq i \leq n$. Assume inductively that this holds until time $t$, and that coordinates $i,j$ are updated at time $t$. Using the representation in \eqref{EqNarrowRepDeltaPos} and \eqref{EqNarrowRepDeltaNeg}
\begin{align*}
X_{t+1}[i,1] & \geq \lambda \min(X_{t}[i,1] + X_{t}[j,1], 2) \\
& \geq \lambda \min(S_{t}[i] + S_{t}[j], 2) \\
& = \lambda (S_{t}[i] + S_{t}[j]) \\
&= S_{t+1}[i]
\end{align*} \par 
Letting $t$ go to infinity, we can consider $S_{\infty}$ and $X_{\infty}$ drawn from the stationary distributions of their respective Markov chains. The bound then follows from Lemma 4.6 of \cite{Smit11}.

\end{proof} \par 

Next, let $p(X,Y,i,j,S)$ be the probability that a subset coupling of $X$, $Y$ associated with subset $S$ works given that coordinates $i,j$ are updated. The proof of the following lemma is nearly identical to the proof of Lemma 4.4 in \cite{Smit11}:

\begin{lemma} [Subset Coupling] \label{LemmaNarrowSubsetCoup}
Assume $n \geq 6$ and fix $a-2 > b > 0$. For a pair of matrices $(x,y)$ satisfying $\sup_{m,k} \vert x[m,k] - y[m,k] \vert \leq n^{-a}$ and $\inf_{m,k} (x[m,k], y[m,k], 2 - x[m,k], 2 - y[m,k]) \geq n^{-b}$, we have for all sufficiently large $n$ that $p(x,y,i,j,S) \geq 1 - 4 n^{b + 2 - a}$ uniformly in $S \subset [n]$ and pairs $i \in S$, $j \notin S$.
\end{lemma}

Next, as in Lemma 9, note that after a successful subset coupling involving sets $S$ and $R$ at time $t$, we have $\vert \vert X_{t+1} - Y_{t+1} \vert \vert_{1,S} \leq \vert \vert X_{t} - Y_{t} \vert \vert_{1,S \cup R}$. Thus, if all subset couplings until time $t$ have succeeded,

\begin{equation} \label{IneqSmallnessNarrowMats}
\vert \vert X_{t} - Y_{t} \vert \vert_{1,A} \leq \vert \vert X_{T_{1}} - Y_{T_{1}} \vert \vert_{1}
\end{equation}

for all $S \in P_{t}$. \par 

We are ready to prove Lemma \ref{LemmaNearbyCoupling}. By Lemma 5, $X_{T} = Y_{T}$ unless at least one subset coupling has failed or $P_{T_{1}} \neq \{ [n] \}$. Let $E_{1}$ be the event that $P_{T_{1}} \neq \{ [n] \}$, let $E_{2}$ be the event that $\inf_{i,j} \inf_{T_{1} \leq t \leq T_{2}} (Y_{t}[i,j], 2 - Y_{t}[i,j]) \geq n^{-\frac{a}{2} - 1.75}$, and let $E_{3}$ be the intersection of $E_{2}^{c}$ with the event that at least one subset coupling fails. By Lemma 4, $P[E_{1}] \leq 2n^{3.75 - \frac{a}{2}}$. By Lemma \ref{LemmaNarrowSubsetCoup} and Lemma \ref{LemmaSmallness}, $P[E_{3}] \leq 4 n^{3.75 - \frac{a}{2}}$. By Lemma \ref{LemmaLargenessMatrix}, $P[E_{2}] \leq 4 n^{3.75 - \frac{a}{2}}$. This completes the proof of Lemma \ref{LemmaNearbyCoupling}. \\

\par 

Finally, we prove Theorem 2. Fix $a = \frac{2}{11}(C + 18.25)$. We divide our run of length $T = C n \log(n)$ into a contractive phase of length $T_{1} = \left( \frac{9}{2}a - 11.25 \right) n \log(n)$ and a coupling phase of length $T_{2} = \left( a - 7 \right) n \log(n)$. Let $E_{1}$ be the event that $\vert \vert X_{T_{1}} - Y_{T_{1}} \vert \vert_{1} > n^{-a}$, and let $E_{2}$ be the intersection of $E_{1}^{c}$ with the event that $X_{T} \neq Y_{T}$. By Lemma \ref{LemmaWeakConMat}, $P[E_{1}] \leq n^{3.75 - \frac{a}{2}}$. By Lemma \ref{LemmaNearbyCoupling}, $P[E_{2}] \leq 10 n^{3.75 - \frac{a}{2}}$. This completes the proof of the upper bound in Theorem 2. \par 

To prove the lower bound, let $\tau$ be the (random) first time at which all $2n$ coordinates have been updated. Then fix the starting position $X_{0}$ of the Markov chain and let $H_{i,j} = \{ X \in \mathbb{R}^{2n} \vert X[i,j] = X_{0}[i,j] \}$ and set $H = \cup_{i,j} H_{i,j}$. Then $P[X_{t} \in H] - U_{n}(H) \geq P[\tau > t]$. Since only four of $2n$ coordinates are chosen at a time, the classical coupon-collector results in \cite{ErRe61} tell us that at time $T = \frac{1}{2} n (\log(n) - c)$, $ \vert K_{n}^{T}(x, H) - \pi(H) \vert \geq 1 - \exp(-\exp(c)) + o(1)$ as $n$ goes to infinity. 

\bibliographystyle{plain}
\bibliography{MoreGibbsBib}

\end{document}